\documentclass[11pt,reqno]{amsart}
\usepackage{amssymb, amsmath}
\usepackage{mathrsfs}
\usepackage{enumerate}

\newcommand {\PP}{{I\kern-.3em P}}
\newcommand {\ZZ}{{\mathbb{Z}}}

\newcommand {\HH}{{\mathbb{H}}}
\newcommand {\RR}{{\mathbb{R}}}
\newcommand {\NN}{{\mathbb {N}}}
\newcommand {\QQ}{{\mathbb{Q}}}

\newcommand{\beq}{\begin{equation}}
\newcommand{\eeq}{\end{equation}}
\newcommand{\beqq}{\begin{equation*}}
\newcommand{\eeqq}{\end{equation*}}

\newcommand{\pmtx}[1]{\begin{pmatrix}#1\end{pmatrix}}
 \numberwithin{equation}{section}

\newenvironment{prolist}%
{\begin{enumerate}[\bfseries 1.]}{\end{enumerate}}
{\begin{enumerate}[\bfseries a.]}{\end{enumerate}}

\def\d{\displaystyle}

\newtheorem{thm}{Theorem}
\newtheorem{lemma}{Lemma}
\newtheorem{prop}{Proposition}
\newtheorem{cor}{Corollary}
\begin{document}
\title{Geodesic cusp excursions and metric diophantine approximation   }
\author{Andrew Haas }
 \address{University of Connecticut\\ Department of Mathematics\\
 Storrs, CT 06269}
 \email{haas@math.uconn.edu}

\begin{abstract}

We derive several results   that describe the rate at which a generic geodesic makes excursions into and out of a cusp on a finite area hyperbolic surface and relate them to approximation with respect to the orbit of infinity for an associated Fuchsian group.
This provides proofs of some well known theorems from metric diophantine approximation in the context of Fuchsian groups. It also gives new results in the classical setting.  
 \end{abstract}

 \subjclass{ 30F35, 32Q45, 37E35,  53D25, 11K60}
 
 \keywords{Hyperbolic surfaces, Fuchsian groups, geodesic flow,  cusp, metric diophantine
 approximation }
\maketitle
%\markboth{The distribution of excursions}
%{The distribution of excursions  }

\section{Introduction}\label{intro}
There is a simple, direct relationship between  classical diophantine approximation and the behavior of geodesics on the Modular surface.  In a number of instances mathematicians have made good use of this sort of  connection in a more general setting: in order to better understanding the geometry of geodesics and other natural geometric objects in certain Riemannian manifolds, to study approximation with respect to  the discrete groups associated with such manifolds as well as to shed light on the classical theory of rational approximation and its generalizations.  For some   examples and further references see \cite{bks, harvey, haas0, her, nakada, series, sullivan}.

 In this paper we continue the investigations begun in \cite{haas} into the particular connection between metric diophantine approximation and the generic behavior of geodesic excursions into a cusp neighborhood on a finite area hyperbolic 2-orbifold. The geometric piece of the earlier paper was concerned with the distribution of the sequence of depths of maximal incursions into the cusp. The distribution exists for a generic geodesic and is independent of the geometry and topology of the surface.  
Here we begin by calculating the asymptotic rate at which a geodesic makes excursions into a cusp neighborhood. This quantity exists for almost all geodesics and, in contrast to the distribution of depths, does depend on the area  of the orbifold. The geometric results are  then applied to derive versions of several well known theorems in classical diophantine approximation in the context of Fuchsian groups, as well as some new theorems in the classical setting. 
It is surprising how many results from classical  diophantine approximation are tied together by the underlying geometry.%This point of view seems to shed new light on the classical results. 
%\\

%\subsection{Background}

In the remainder of this section we  establish some notation and state the main results of the paper.
 \subsection{Geodesic cusp excursions}
 Let $S$ be a finite area, non-compact hyperbolic 2-orbifold and let $P$ denote a
non-compact end, often called a puncture, of $S$. 
There is a canonically defined
horocyclic, cusp neighborhood  of $P$ which we
shall call $C_{k}.$ If  $S$ is represented by a  Fuchsian group $\Gamma$ so that the stabilizer of infinity is generated by the transformation $T(z)=z+1$ and $P$ corresponds to the point at infinity, then $C_k$ is the projection to $S$ of the half-plane $H_{\frac{1}{k}} =\{z\,|\, \text{Im}(z)>\frac{1}{k}\}.$
If $k<1$ then  the area of $C_k$ is  always $k$. Such a normalized group $\Gamma$ is called a {\em zonal} Fuchsian group

Given a unit tangent vector $ {\mathbf v}\in
T_1S$, let $\gamma({\mathbf v})=\gamma$ denote the unique geodesic on $S$ in the direction of ${\mathbf v}.$   An   excursion of $\gamma$ into $C_k$,  called a $k$-{\it excursion}, is a point $\gamma(\tau)$, on the geodesic $\gamma$, at which $\gamma$ is tangent to a cusp neighborhood $C_d$ for $d< k.$  The value $\tau\in\RR$ is called the {\em depth parameter} for the $k$-excursion and $d$ is the {\em depth of the excursion}. For $k\leq 2$, a $k$-excursion is an {\em approximating} $k$-excursion if there is a \lq\lq loop" of $\gamma$ about $P$, corresponding to the excursion. The precise definition is given in Section \ref{normalize}. The motivation for distinguishing this particular type of excursion will become clear  in Section \ref{approximation}, where we see how approximating excursions correspond to the continued fraction  convergents in classical diophantine approximation. It is worth noting that when $k\leq1$ every  $k$-excursion is an approximating $k$-excursion. 

In dealing with  $k$-excursions it is necessary for us to place some restrictions on the orbifolds being considered.   Our assumptions will be spelled out in Section \ref{normalize}. No such restrictions are necessary when dealing with approximating excursions or with excursion on surfaces without order-2 cone singularities.
 
The depth parameters for the $k$-excursions along a geodesic are called the {\em $k$-depth parameters}. They can be indexed by $\NN$ in such a way that they respect the ordering as real numbers.  Let $\{t_i\}$ denote the ordered sequence of $k$-depth parameters along a geodesic $\gamma.$ It is easy to see that $t_i\rightarrow\infty$ as $i\rightarrow\infty$.
For all ${\mathbf v}$ in a full measure subset $\mathcal{E} $ of $T_1S$ and for every $0< k\leq 2$, there are infinitely many   $k$-excursions $\gamma(t_i)$ as well as infinitely many approximating $k$-excursions $\gamma(t_{i_j})$ in the forward direction along $\gamma$, \cite{haas}.

The set of $k$-depth parameters is written $\Pi_{{\mathbf v}}(k)=\{ t_i\geq 0\,|\,   \gamma({\mathbf v})(t_i)\, \text{is a } \, k-\text{excursion}\}.   $ Let $\Pi^a_{{\mathbf v}}(k)$ be the subset  corresponding to approximating excursions. The notation $\#X$ shall denote the cardinality of the set $X$. 
For  $0<k\leq 2$ the value  $N_{{\mathbf v}}(k)(t)=\#\{0\leq t_e\leq
t|\, t_e\in\Pi_{{\mathbf v}}(k)\}$  is defined for all $t\geq 0$. Similarly, define $N_{\mathbf v}^a(k)(t)=\#\{0\leq t_e\leq
t|\, t_e\in\Pi_{\mathbf v}^a(k)\}.$ $N_{\mathbf v}$ and $N_{\mathbf v}^a$, respectively, count the number of $k$-excursions and approximating $k$-excursions as  functions of arc length.  We shall let the symbol $*$ denote either the letter $a$ or nothing. For example,  $N^*_{\mathbf v}$  is either $N_{\mathbf v}$ or $N_{\mathbf v}^a$.  Define the function

 \beqq
 \d{{\mathcal A}}(z)=
\left\{ \begin{array}{ll}
  z & \mbox{if $0<z\leq 1$}\\
 (2-z+2\log z)  & \mbox{if $ 1\leq z\leq2.$ }
\end{array}
\right.
\eeqq

Then we have

\begin{thm}\label{counting}
For almost all ${\mathbf v}\in T_1S$ and $0<k\leq 2$
\beq\label{overlap}
\lim_{t\rightarrow\infty}\frac{N_{\mathbf v}(k)(t)}{t}=\frac{k}{\pi\, \text{area}(S)}
 \eeq
and
\beq\label{app}
\lim_{t\rightarrow\infty}\frac{N_{\mathbf v}^a(k)(t)}{t}=\frac{  \mathcal{A}(k) }{\pi\, \text{area}(S)}.
  \eeq
  \end{thm}
 
 Our proof employes  a technique developed by Nakada \cite{nakada} and Stratmann \cite{strat} to count the rate at which an orbit returns to a cusp. Their idea was to define a cross-section to the flow and thicken it. Each return to the cusp is counted by a return to the thickened section and then  the ergodicity of the geodesic flow is used to determine the asymptotic rate of return. Our approach makes us of a slightly more elaborate cross-section to the flow which allows us to count, not only excursions into simple cusp neighborhoods, but approximating excursions as well as excursions into neighborhoods that overlap themselves (the area of $C_k$ is not k).   Even when there is overlap, as in the case of the Modular Surface,  the limit in (\ref{overlap}) is proportional to $k$.  
  
Since geodesics are parameterized by arc length, the distance between consecutive excursions is $ t_i-t_{i-1}$. Suppose $k\leq1$. Associated to each depth parameter $t_i$ there are boundary parameters $\delta_{2i-1}$ and $\delta_{2i}$ so that the arc of $\gamma(t)$ for $\delta_{2i-1}\leq t\leq \delta_{2i}$ is the maximal arc of $\gamma$ in $C_k$ containing $\gamma(t_i)$. It makes sense to define the length of the excursion $\gamma(t_i)$ to be  $\delta_{2i}-\delta_{2i-1} $ An immediate consequence of the theorem is

\begin{cor}\label{return}
For almost all $v\in T_1S$
\begin{itemize}
\item[{\rm (i)}]
For $k\leq 2$ the average length of an arc of $\gamma$ between consecutive $k$-excursions or approximating $k$-excursions is, respectively, $\pi \text{area}(S)\,k^{-1}$ or\\ $\pi \text{area}(S)\,\mathcal{A}(k)^{-1}. $
%\beqq
%\lim_{n\rightarrow\infty}\frac{1}{n}\sum_{i=1}^n t_i-t_{i-1},
%\eeqq
%is
%given by the inverse of the formulas of Theorem \ref{counting}.
\item[{\rm (ii)}] For $k\leq 1$ the average  length of a $k$-excursion is $\pi.$
%\beqq
%\lim_{n\rightarrow\infty}\frac{1}{n}\sum_{i=1}^n \delta_{2i}-\delta_{2i-1}=\pi.
%\eeqq
 
\end{itemize}
\end{cor}

The assumption that $k\leq 1$  in Part (ii) is necessary to guarantee that the lifts of $C_k$ are interior disjoint, and therefore, the definition of the boundary parameters makes sense. If the order two elliptic $U_{\eta}$, defined in section \ref{normalize} does not belong to $\Gamma$ then the corollary holds for $k\leq 2$. In fact, using the methods of \cite{strat} the corollary can be proved for all $k$ where the $C_k$ are interior disjoint.

Note  that the average in Part (ii) of the corollary is independent of the parameter $k$. 
%It should be emphasized that $\pi$ is not actually the \lq\lq average value" of an excursion into $C_k$, although that value is also independent of $k$, \cite{choi}.
 % Part (ii) is also  equivalent to saying that the average time it takes a geodesic to return to the cusp $C_k$ is
%$\pi( \text{area}(S)/k)\,-\pi.$ 
 
\subsection{Approximation by the orbit of infinity} 
Suppose the zonal Fuchsian group $\Gamma$ uniformizing $S$ is   further normalized, as  described in Section \ref{normalize}.  The orbit of infinity is   the set of $\Gamma$-{\em rational} numbers, written $\QQ(\Gamma)$, which naturally appear as \lq\lq fractions" in entries from elements of $\Gamma.$ For almost all $x\in (0,1)$ there
exist two special sequences of $\Gamma$-rational numbers
converging to $x$. The first consists of all fractions $p_n/q_n$   that 
 satisfy the familiar equation
\begin{equation}\label{q^2}
|x-\frac{ p_n}{ q_n}|< \frac{1}{ q_n^2}
\end{equation}
and are called the {\em $\Gamma$-n-convergents}. The second type, simply called the {\em $\Gamma$-convergents} of $x$ are contained as a subsequence of the first and consist of fractions $p_n^a/q_n^a$ that satisfy a condition generalizing the definition of the sequence of convergents coming from the continued fraction expansion of an irrational number.
If $\Gamma=\text{PSL}_2(\ZZ)$,  the classical modular group, then the fractions $  p_n^a/q_n^a$ are precisely the classical convergents to $x$, whereas the fractions $p_n/q_n$ are the   convergents and a subset of the nearest mediants \cite{bosma}. This will all be made more precise in the Section \ref{approximation}.

Translating the \lq\lq length between excursions" into the language of   approximation by $\Gamma$-rationals   we get:
% Theorem \ref{counting} translates into the following and also making use of Theorem \ref{theta}:

\begin{thm}\label{0thm}
For almost all $x\in(0,1)$  
\beq\label{conv}
\lim_{n\rightarrow\infty}\frac{\log q^a_n}{n}= -\frac{1}{2}\lim_{n\rightarrow\infty}\frac{1}{n}\log |x-\frac{p^a_n}{q^a_n}|= \frac{ \pi\, \text{area}(S)}{4\log 2}
\eeq
and
\beq\label{conv2}
\lim_{n\rightarrow\infty}\frac{\log q_n}{n}=-\frac{1}{2}\lim_{n\rightarrow\infty}\frac{1}{n}\log |x-\frac{p_n}{q_n}|=\frac{1}{4}\pi\, \text{area}(S).
\eeq
 
\end{thm}
The first limit (\ref{conv}) is due to Levy  in the classical case of continued fraction convergents \cite{bjw, sinai}. Since the area of the modular surface is $\pi/3$, the limits (\ref{conv}) take the value $\pi^2/(12\log 2)$.  This is the the  Khinchin-Levy constant. From our point of view, this constant is one half the average distance between approximating excursions on the modular surface.
The second limit (\ref{conv2}) of the theorem appears to be new even in the classical case. A proof using more traditional methods would be welcome.

 Given $x\in (0,1)$ and any one of the sequences of $\Gamma$-approximants  $p^*_n/q^*_n$   defined earlier, define the associated approximation constants  $\theta^*_n(x) = q^*_n|q^*_nx-p^*_n|$, \cite{bosma, bjw, haas}.   Then it is possible to show that
 \begin{thm}\label{theta}
For almost all $x\in (0,1)$
\begin{equation*}
\limsup_{n\rightarrow\infty} -\frac{\log \theta^*_n(x)}{\log n}=1.
\end{equation*}
\end{thm}

The proof involves a discretization of Sullivan's logarithm law for geodesics \cite{sullivan}. It should be emphasized that this result will be needed to equate the limits in (\ref{conv}) and (\ref{conv2}) and complete the proof of Theorem \ref{0thm}.
In the classical setting the theorem can be further translated into a result about the sequence of partial quotients of a generic continued fraction. More specifically,  suppose $x\in (0,1)$ has the uniquely
defined,  infinite continued fraction expansion 
$$x=[a_1, a_2, \ldots]= 
  \frac{1}{a_1+\frac{1}{a_2+\frac{}{a_3+ \dots}}}.
$$
   The positive integers $a_i$ are called the partial quotients of the continued fraction expansion  \cite{HW}. The convergents are $\frac{p^a_n}{q^a_n}=[a_1,\ldots, a_{n}].$ 
   %Then we see that a classical version of Sullivan's Theorem has the form

\begin{cor}\label{classical}
For almost all $x=[a_1, a_2,\ldots]\in (0,1)$
\begin{equation*}
\limsup_{n\rightarrow\infty} \frac{\log a_n}{\log n}=1.
\end{equation*}
\end{cor}
%The corollary sharpens the well know fact that for almost all $x$
%\beqq
%\lim_{n\rightarrow\infty} \frac{\log a_n}{  n}=0.
%\end{equation*}
 %It has been suggested that this result might be derived using more standard probabilistic techniques \cite{hensley}.
 The corollary can also be derived using  more traditional methods. We shall describe both proofs  in Section \ref{cor pf}.

Theorem \ref{counting} can also be used to derive the following averages. 

\begin{thm}\label{thetasums}
For almost all $x\in (0,1)$
\begin{itemize}
\item[{\rm (i)}]
\beqq\label{theta1}
\lim_{ n \rightarrow\infty}\frac{1}{n}\sum_{i=1}^n  \theta_i^*(x)=
\left\{ \begin{array}{ll}
1/2 & \mbox{if\, *=nothing}\\
 1/(4\log 2)  & \mbox{if *=a }
\end{array}
\right.
\eeqq
\item[{\rm (ii)}]

\beqq\label{a}
\lim_{ n \rightarrow\infty}\frac{1}{n}\sum_{i=1}^n  \log \theta_i^*(x)= 
\left\{ \begin{array}{ll}
- 1 & \mbox{if\, *=nothing}\\
-\frac{1}{2}\log 2 -1  & \mbox{if *=a }
\end{array}
\right.
\eeqq
\end{itemize}
\end{thm}

   In the classical case, all of these averages  either appear explicitly in the literature or are easily derived from know results, \cite{bosma, bjw, haas1}.
\\

\subsection{The structure of the paper}  
The paper is organized as follows. At the start of Section \ref{section2} the basic notions and definitions that were introduced in Section \ref{intro} shall be made more precise. We then define sections of the unit tangent bundle that count excursions and approximating excursions along geodesics. The relative time a geodesic spends in a thickened version of one of the sections is shown to be proportional to the rate at which it makes excursions or approximating excursions into the cusp. In this way, after computing the area of the thickened section, Theorem \ref{counting} is proved. In the remainder of the section we prove Corollary \ref{return} and then begin laying the foundation for the proof of Theorem \ref{0thm}.
In Section \ref{xdepth} we prove the geometric versions of Theorems \ref{theta} and  \ref{thetasums}.  Finally, in Section \ref{approximation} the geometric results are applied to the problem of $\Gamma$-rational approximation. \\

\section{Geodesic Excursions}\label{section2}

 \subsection{Group normalization, geodesics  and excursions}\label{normalize}

Recall the definition of the horocycle  $H_{m } =\{z\,|\, \text{Im}(z)>m\}$.  Let $\eta=\eta(\Gamma)$ be the smallest number so that $H_{ \eta }\cap g(H_{ \eta })=\emptyset$ for $g\not \in \text{Stab}(\infty)$, the stabilizer of infinity in $\Gamma.$  Normalize the group so that there exists $h\in \Gamma$ with $h(\infty)=0$  and so that the closures of the horocycles $H_{ \eta }$ and    $h(H_{ \eta })$  are tangent at a single
    point $ \eta\text{i}$ for i$=\sqrt{-1}.$ 
      Then  there is  no overlap and the area of $C_k$ is $k$ if and only if $k<\frac{1}{\eta}$. It is well known that $ \eta\leq 1$ and one can show that for any group other than the classical modular group $\eta<1.$ Furthermore, if $\Gamma$ does not contain an order two elliptic fixing  $\eta\text{i}$, then it is easy to prove that $\eta\leq 1/2.$
     
 $\Gamma$ acts on the hyperbolic upper-half space $\HH$ so that the quotient $\HH/ \Gamma$ is the 2-orbifold $S$. Let $\pi$ denote the quotient map and let $\pi_*$ be the induced map from the unit tangent bundle $T_1\HH$ to $T_1S.$ Similarly, $g\in \Gamma$ induces an isometry $g_*$ of $T_1\HH$. This use of the subscript $*$ is distinct from its use as a superscript, where it denotes either $a$ or nothing.
 
  Let   $\tilde{l}_{\tau}=\{z\,|\,\text{Re}(z)=\tau\}$ and write  $\tilde{l}$ for $\tilde{l}_0$. Let $\pi(\tilde{l})=l$ be the projection to $S.$
Note that because of our normalization of $\Gamma$, $l$ is simple.
Set $\mathscr{L}=\{g(\tilde{l})\, |\, g\in\Gamma\},$ the full set of lifts of $l$ to $\HH.$
 
 With $\eta=\eta(\Gamma),\, U_{\eta}(z)=-\frac{\eta^2}{z}$  is the order two elliptic, fixing $i\eta$, interchanging 0 and $\infty$ and taking $\tilde{l}$ to itself. If $U_{\eta}\in \Gamma$ then the only geodesics in $\mathscr{L} $ with an endpoint at infinity have the form $\tilde{l}_m,\,m\in\ZZ$. If $U_{\eta}\not\in\Gamma$ then there exists a  number $\tau_0,$ which is not an integer, so that $\tilde{l}_{\tau}\in\mathscr{L}$ if and only if $\tau=m$ or $\tau=\tau_0+m$ for $m\in\ZZ$. In particular, if $g\in\Gamma$ takes 0 to infinity then $g(\tilde{l})=\tau_0+m$ for $m\in\ZZ$. Henceforth, we suppose that $\tau_0$ is the smallest positive number for which the above holds. If $U_{\eta}\in \Gamma$ then one of the endpoints of $l$ is a cone point, which is the projection of the fixed point of $U_{\eta}.$

Let   $\gamma:(-\infty,\infty)\rightarrow S$  be a geodesic  parameterized by arc length.
A lift $\tilde{\gamma}$ of $\gamma$ to $\HH$
  has endpoints $\tilde{\gamma}_+$ and   $\tilde{\gamma}_-,$
corresponding to $\tilde{\gamma}(t)$ as $t\rightarrow \infty$ and  $t\rightarrow -\infty,$ respectively. 
  If the domain of $\gamma$ is restricted to $[0,\infty)$ then we shall
refer to it as a geodesic ray, which then has the single asymptotic endpoint $\tilde{\gamma}_+$ at infinity.
 
%In order to make the depth of a $k$-excursion a well defined quantity, we expand the  definition of the set of tangent vectors $\mathcal{E}$. To this end we shall henceforth suppose that if ${\mathbf v}\in\mathcal{E}$ then $\gamma =\gamma({\mathbf v})$ (1) does not contain a point $p$ at which $\gamma$ is tangent to cusp neighborhoods $C_{d_1}$ and $C_{d_2}$, for $d_1 <d_2 <k$ and (2) if $\tilde{\gamma}$ is a lift of $\gamma$ which is tangent to the boundary of the horocycle $H_{\frac{1}{d}}$, then every other lift of $\gamma$ which it tangent to $H_{\frac{1}{d}}$ is a translate of  $\tilde{\gamma}$ by an element of $\text{Stab}(\infty)$.  It is easy to show that the excluded set of vectors is a two dimensional subset of the unit tangent bundle and therefore has measure zero in the natural measure on $T_1(S)$; see Section \ref{its inv}.

In order to clarify the definition of a $k$-excursion and it's depth,  we  further restrict the set of tangent vectors $\mathcal{E}$.    To this end, let $v\in T_1(S)$ and set $\gamma =\gamma({\mathbf v})$. Suppose there are numbers $d_1\leq d_2< k$ and lifts $\tilde{\gamma}_1$ and      $\tilde{\gamma}_2$  of $\gamma$, tangent to $H_{\frac{1}{d_1}}$ and     $H_{\frac{1}{d_1}}$ respectively, at   lifts of a fixed excursion $e$. If   ${\mathbf v}\in\mathcal{E}$, then we shall require that  $d_1=d_2$ and  $\tilde{\gamma}_2 =g(\tilde{\gamma}_1) $ for some $g\in \text{Stab}(\infty)$.  This essentially says that $\gamma$ cannot be tangent to two horocyclic neighborhoods in $C_k$ at the same point.   It is easy to show that the excluded set of vectors is a two dimensional subset of the unit tangent bundle and therefore has measure zero in the natural measure on $T_1(S)$; see Section \ref{its inv}.

 Define the two planar sets $J=[(-1,0)\times (0,\infty)]\cup
[(0,1)\times (-\infty,0)]$ and $I=[(-1,0)\times (1,\infty)]\cup
[(0,1)\times (-\infty,-1)].$

Given a $k$-excursion $e=\gamma(t_e)$ of the geodesic $\gamma$ with $k\leq 2$, define  $\tilde{\gamma}=\tilde{\gamma}_e$  to be the   
unique lift
%\beq\label{game}
%\tilde{\gamma}=\tilde{\gamma}_e
%\eeq
of $\gamma$ to $\HH$
so that $\tilde{\gamma}$ is tangent to $H_{\frac{1}{d(e)}}$ at a
point $\tilde{\gamma}(t_e)=\tilde{e}$ and $(\tilde{\gamma}_+, \tilde{\gamma}_-)\in J.$  

  The $k$-excursion $e$ is defined to be an {\it approximating} $k$-excursion if
  $( \tilde{\gamma}_+,
\tilde{\gamma}_-)\in I.$     In the classical case,  the approximating excursions into the cusp on the modular surface play a significant role through their association with the continued fraction convergents. This was a recurrent theme in \cite{haas} and we shall return to it in Section 4. Observe that an $h$-excursion for $h\leq 1$, is   an
approximating $k$-excursion for every $h\leq k\leq 2.$

When   considering $k$-excursions  we will assume that, except in the case where $\Gamma=\text{PSL}_2(\ZZ)$, the involution $U_{\eta}\not\in\Gamma$.  This is the restriction on $\Gamma$ referred to in the introduction. Unfortunately, this forces us to exclude some interesting examples, such as the Hecke groups, \cite{Beardon}. It would be interesting to see what form (\ref{conv2}) might take for these omitted cases. Otherwise, for approximating excursions absolutely no such restrictions are necessary.

The $k$-excursions along $\gamma$ are naturally ordered by defining  $e_1<e_2$ if $t_{e_1}<t_{e_2}.$
 Since the depth parameters diverge to infinity, excursions cannot accumulate \cite{haas}. 
 This also induces an ordering on approximating $k$-excursions.
We shall also  be interested
in the parameter values $s_e$ for which $\tilde{\gamma}_e(s_e)=\tilde{\gamma}_e
\cap \tilde{l}.$  The value $s_e$ is called the {\it excursion
parameter} for the approximating $k$-excursion $e$. It is not hard to see that these values are well
defined and that they are ordered like the corresponding
  excursions.

Let $e_i$ be a  $k$-excursion along a geodesic $\gamma$ with depth and excursion parameters  $t_i$ and $s_i.$ Suppose there is a  $k$-excursion $e_{i-1}$ preceding $e_i$ in the natural ordering of $k$-excursions, with excursion parameter $s_{i-1}.$ Then one can show that $\max\{t_{i-1},s_{i-1}\} < \min\{t_i,s_i\}.$

 We end the section with a lemma and a definition.
 \begin{lemma}\label{ab}
Fix $0<a<b$ and suppose $\alpha_x$ is the geodesic in $\HH$ with endpoints 0 and $x>b$. Let $D(a,b,x)$ denote the distance between $\tilde{l}_a$ and $\tilde{l}_b$ along $\alpha_x$. Then $\inf_x D(a,b,x)=\delta(a,b)>0.$  
\end{lemma}
\begin{proof}
Since $D(a,b,x)\rightarrow\infty$ as $x\rightarrow b$, it will suffice to show that $D(a,b,x)$ is bounded below as $x\rightarrow \infty$. The points at which $\alpha_x$ meets  $\tilde{l}_a$ and $\tilde{l}_b$ respectively are,
$a+i\sqrt{ax-a^2}$ and $b+i\sqrt{bx-b^2}.$ Thus for $x$ sufficiently large, the vertical distance in $\HH$ between the points is $\frac{1}{2}\log |\frac{bx-b^2}{ax-a^2}|.$ This is an increasing function of $x$ and a lower bound for the distance along $\alpha_x.$ Thus  $\delta(a,b)>0.$   
\end{proof}

%We will make use of the elementary fact that if $0<a<b<c,$ then $\delta(a,b)<\delta(a,c).$
When   $U_{\eta}\not\in \Gamma$ define the number $\delta(\Gamma)$ to be one-half the minimum of the values $\delta(1-\tau_0,1), \delta(1,1+\tau_0), \delta(1,2-\tau_0), \delta(1,2),$ $ \delta(2,3), \delta(3,4), -\log (\sqrt{2}-1) .$ When   $U_{\eta}\in \Gamma$ define   $\delta(\Gamma)$ similarly, except remove the terms containing $\tau_0$ in the above.

\subsection{The geodesic flow and its invariant measure}\label{its inv}
The unit tangent bundle $T_1\HH$ is identified with $\HH \times
S^1.$ The measure $dAd\theta$ is an invariant measure for the
geodesic flow $\tilde{G}_t$ acting on $T_1\HH.$ If $\text{area}(S) $ denotes
the area of the orbifold $S$, then the normalized measure
%$\tilde{\mu}=\frac{1}{2\pi \text{area}(S)}dAd\theta$ 
$\tilde{\mu}=(2\pi\, \text{area}(S))^{-1}dAd\theta$
projects to a probability
measure $\mu$ on $T_1S$ that is invariant with respect to the
geodesic flow $G_t$ on $T_1S$. It is known that $G_t$ acts
ergodically with respect to $\mu,$ \cite{nicholls}. 

Recall that $\mathcal{E}$ is the set of   ${\mathbf v}\in T_1S$ for which 
$\gamma({\mathbf v})$ contains infinitely many   $k$-excursions for
every $k>0$ and let $\tilde{\mathcal{E}}$ be the set of lifts of vectors in $\mathcal{E}$ to $T_1\HH.$   One important consequence of the Poincare recurrence and   Ergodic Theorems is that
$\mathcal{E} $ is a set of full measure in $T_1S$ \cite{haas}. In fact, we can take $\mathcal{E}$ to have the even stronger property that for $v\in \mathcal{E}$, $\gamma(v)$  contains infinitely many approximating $k$-excursions for every $k>0$, \cite{haas}.

There are other useful coordinates, defined on a full measure subset of $T_1\HH$, consisting of triples
$(x, y, t)\in \RR^3$ with $x\not =y.$ Let $\alpha$ be
the geodesic with $\alpha_+=x,\, \alpha_-=y,$   so that
$\alpha(0)$ is the Euclidean midpoint of the semicircle
$\alpha(\RR).$ Then $(x,y, t)$ corresponds to the vector
$ {\mathbf v}= \dot{\alpha}(t)\in T_1\HH.$ Several things become
particularly nice in  these coordinates. The geodesic flow has the
simple form $\tilde{G}_t(x, y, s)=(x, y, s+t).$  Furthermore
$\tilde{\mu}=\tilde{\nu}\times dt,$ where  
%$\tilde{\nu}=\frac{1}{\pi \text{area}(S)}(x-y)^{-2}dx dy$
$\tilde{\nu}=(\pi\, \text{area}(S))^{-1}(x-y)^{-2}dx dy$
\cite{nicholls}.  

\subsection{Cross sections of the flow that count excursions}

 For $k>0$, define the set  
 $\Omega_k=\{(x,y)\in J\,|\,  |x-y|>2/k\}$
    Also, let   $\Delta_k=\Omega_k\cap I.$ 
 
We shall define a cross section  $\mathcal{L}^*_k$ for the geodesic flow
  on $T_1S$ first by defining a subset $ \tilde{\mathcal{L}}^*_k$ of the
unit tangent bundle over $\tilde{l}.$ Let $
\tilde{\mathcal{L}}_k=\{(x,y,t)\in T_1\tilde{l}\cap\tilde{\mathcal{E}} |\, (x,y)\in \Omega_k\}$ and
write $\pi_*(\tilde{\mathcal{L}}_k)=\mathcal{L}_k\subset T_1l.$  Let $
\tilde{\mathcal{L}}_k(\epsilon)=\{G_t(x,y,s)|\, (x,y,s)\in
\tilde{\mathcal{L}}_k\, \text{and}\, t\in [0,\epsilon)\}.$
Similarly, define  $\tilde{\mathcal{L}}_k^a=\{(x,y,t)\in \tilde{\mathcal{L}}_k \, |\, (x,y)\in I \} $ and set $\mathcal{L}_k^a=\pi_*(\mathcal{L}_k^a).$ The cross sections $\tilde{\mathcal{L}}_k$ and $\tilde{\mathcal{L}}^a_k$ are the same as the ones defined in \cite{haas}.\\
\\  
\begin{prop}\label{thickL}
With $k<3$ for excursions and $k\leq 2$ for approximating excursions and $\epsilon \leq \delta(\Gamma)$
\begin{itemize}
\item[{\rm (i)}]\label{i} If $g\in \Gamma,\, g\not =\text{Id}$  and $ \tilde{{\mathbf v}}\in \tilde{\mathcal{L}}^*_k(\epsilon)$
%g_*(\tilde{\mathcal{L}}_k(\frac{1}{2}\log 2))\cap
%\tilde{\mathcal{L}}_k(\frac{1}{2}\log 2)=\emptyset.$\\
then $g_*(\tilde{{\mathbf v}})\not\in \tilde{\mathcal{L}}^*_k(\epsilon )$\\
\item[{\rm (ii)}] $ \mathcal{L}^*_k  $ is a cross-section for the geodesic flow on a set of
full measure in $T_1S.$
\end{itemize}
\end{prop}
Part (i) says that the projection from $\tilde{\mathcal{L}}_k(\delta(\Gamma))  $ to $\mathcal{L}_k(\delta(\Gamma))  $ is 1-to-1. Part (ii) is proved in \cite{haas} and says that as you flow along a generic orbit of the geodesic flow, you repeatedly return to $\mathcal{L}_k  $. Each time you do corresponds to an excursion of the geodesic, which is the projection of that orbit.

Let $T^{\tau}(z)=z+\tau$ and define $R^{\tau}=U_{\eta}\circ T^{\tau}\circ U_{\eta}.$  
We will be needing the following lemma.
\begin{lemma}\label{stab0}
$R=R^1$ generates the stabilizer of zero,  Stab$(0)$, in $\Gamma$ and $|R^n(\infty)|\leq\frac{1}{|n|}$ for $n\in \ZZ\setminus\{0\}$ with equality if and only if $\Gamma=\text{PSL}_2(\ZZ).$

\end{lemma}
\begin{proof}
Choose $g\in \Gamma$ so that $g(0)=\infty$ and $\tau\in \RR$ so that $T^{\tau}(g(\infty))=0.$ By our normalization  Im$(g(i\eta))=i\eta.$ Then since $T^{\tau}\circ g$ maps $\tilde{l}$ onto itself, it must fix $i \eta$ and it is therefore equal to $U_{\eta}.$ It follows that $h$ generates Stab(0) if and only if $g\circ h\circ g^{-1}$ generates Stab$(\infty)$ if and only if 
$T^{\tau}\circ (g\circ h\circ g^{-1})\circ T^{-\tau}=U_{\eta}\circ h\circ U_{\eta}$ generates Stab$(\infty)$. Thus $U_{\eta}\circ h\circ U_{\eta}=T^{\pm 1}$ which implies $h=R^{\pm 1}.$ 

 As to the last assertion of the lemma,  observe that $R^{\tau}(\infty)=U_{\eta}\circ T^{\tau}(0)$ and therefore $U_{\eta}(R^{\tau}(\infty))=T^{\tau}(0)=\tau.$ Consequently,
$|R^n(\infty)|=|U_{\eta}(n)|=|\frac{\eta^2}{n}|\leq \frac{1}{|n|}$. Since $\eta\leq 1$, there is equality if and only if $\eta=1$, which is precisely when $\Gamma=\text{PSL}_2(\ZZ).$
\end{proof}

\noindent {\em Proof of Proposition \ref{thickL}.}
First let's see that $g(\tilde{\mathcal{L}}^*_k)\cap\tilde{\mathcal{L}}^*_k
=\emptyset.$ If $U_{\eta}\not\in\Gamma$ then this is clear, so suppose $U_{\eta} \in\Gamma$. For approximating $k$-excursions, if ${\mathbf v}=(x,y,s)\in\tilde{\mathcal{L}}^a_k$ then $(x,y)\in I.$ $(U_{\eta})_*({\mathbf v})$ has first coordinates $(U_{\eta}(x),U_{\eta}(y))$. If, for example, $x\in (0,1)$ then   $(U_{\eta}(x),U_{\eta}(y))\in (-\infty, -\eta^2)\times(0,\eta^2)$ which is disjoint from $I$. Therefore,  $(U_{\eta})_*({\mathbf v})\not\in\tilde{\mathcal{L}}^a_k.$ For $k$-excursions we have assumed that $U_{\eta}\in\Gamma$ only if $\eta=1$ and $\Gamma=PSL_2(\ZZ).$ If ${\mathbf v}=(x,y,s)\in\tilde{\mathcal{L}}_k$ then $(x,y)\in J.$ Again considering $x\in (0,1)$,  $(U_{1})_*({\mathbf v})$ has first coordinates 
 $(U_{1}(x),U_{1}(y))\in (-\infty, -1)\times(0,1)$ which is disjoint from $J$. Therefore,  $(U_{1})_*({\mathbf v})\not\in\tilde{\mathcal{L}}_k.$

Set $\Gamma_0=\Gamma\setminus \{\text{id}, U_{\eta}\}$. In order to prove the proposition, we will argue that for ${\mathbf v}\in \tilde{\mathcal{L}}^*_k$
and $t\in (0,\delta(\Gamma)),\, \tilde{G}_t({\mathbf v})\not \in g_*(\tilde{\mathcal{L}}^*_k)$ for any $g\in \Gamma_0$. To this end, suppose ${\mathbf v}\in\tilde{\mathcal{L}}^*_k$. Let $\gamma=\gamma({\mathbf v})$ and suppose that the tangent to $\tilde{\gamma}(t)$ lies in $g( \tilde{\mathcal{L}}^*_k ).$ Then we show that the distance along $\tilde{\gamma}$ between $\tilde{l}$ and $g(\tilde{l})$ is bounded below by $\delta(\Gamma)$.

There are three cases to consider. In the first  $g(\tilde{l})$ and $\tilde{l}$
 do not share an endpoint. Set $D_{\eta}=U_{\eta}(H_{\eta}).$ Then for $g\in\Gamma_0,$    $g(H_{\eta}\cup D_{\eta} )\cap  (H_{\eta}\cup D_{\eta}  )=\emptyset $.  
  The minimal distance between an arbitrary  geodesic $\tilde{\alpha}$, disjoint from $H_{\eta}\cup D_{\eta}$, and $\tilde{l}$ is realized when
$\tilde{\alpha}$ is tangent to both $H_{\eta}$ and $D_{\eta}$. An easy computation (renormalize by taking $(0, \infty, i\eta)$ to $(-1,1,\frac{1}{2}i)$) shows that this distance is $-\log(\sqrt{2}-1).$ The first case follows.

Now we consider the cases where $g(\tilde{l})$ and $\tilde{l}$
 do   share an endpoint. Suppose without loss of generality, that the geodesic $\tilde{\gamma}=\tilde{\gamma}(\tilde{{\mathbf v}})$ for $\tilde{{\mathbf v}} \in \tilde{\mathcal{L}}^*_k$ has endpoints $(\gamma_+,\gamma_-)\in (0,1)\times(-\infty,0).$  If $g(\tilde{l})$ and $\tilde{l}$ have the point at infinity in common, then $g(\tilde{l})=\tilde{l}_{\tau_0}$. The distance between $\tilde{l}$ and $\tilde{l}_{\tau_0}$ is bounded below by $\delta(1-\tau_0, 1)<\delta(\Gamma),$ which leaves us with the case in which  $g(\tilde{l})$ and $\tilde{l}$
 share an endpoint at 0.

 Since the tangent to $\tilde{\gamma}$ at $\tilde{\gamma}\cap g(\tilde{l})$ lies in  $g_*(\tilde{\mathcal{L}}^*_k)$, the tangent to $g^{-1}(\tilde{\gamma})$ at  $g^{-1}(\tilde{\gamma})\cap \tilde{l}$ lies in $\tilde{\mathcal{L}}^*_k$. For this to be true we must at least have $(g^{-1}(\gamma_+),g^{-1}(\gamma_-))\in  J$.

 There are two subcases. 
 \begin{prolist}
 \item[{\rm 1.}] In the first we suppose that $g(0)\not= 0$. Then  $g(\infty)=0$ and there must exist $\tau>0$ so that $g^{-1}(\tilde{l})=\tilde{l}_{\tau}$ where $\tau=\epsilon\tau_0+m-\epsilon$ with $m>0$ and $\epsilon = 0$ or 1 depending on whether $U_{\eta}\in \Gamma$ or not. It follows that $g^{-1}(\gamma_+)\in (-1,0).$ The distance between $\tilde{l}$ and $g(\tilde{l})$  along $\gamma$ is equal to  the distance between $\tilde{l}$ and $\tilde{l}_{\tau}$ along $g^{-1}(\gamma)$. This value is  bounded below by the distance along the geodesic $\alpha$ with endpoints $-1$ and $\tau+1$ between $\tilde{l}$ and $\tilde{l}_{\tau}.$ By Lemma \ref{ab}, this value is bounded below by $\delta(1, 1+\tau ).$ Then we have a lower bound of either  $\delta(1,2)$ or $\delta(1, 1+\tau_0 )$, depending on whether $U_{\eta}\in \Gamma$ or not. Note that, if we had assumed instead at the start, that $\gamma_+\in (-1,0)$ then $g^{-1}(\gamma_+)\in(0,1)$ and the lower bound is $\delta(1, 2-\tau_0)$ or $\delta(1,2)$. That finishes the first subcase.
 
 \item[{\rm 2.}] Now we suppose that $g(0)=0.$ Then by Lemma \ref{stab0}, $g=R^{-n}$ for some integer $n>0.$ The distance along $\gamma$ between $\tilde{l}$ and $g(\tilde{l})$ is bounded below by the distance along $\gamma$ between $\tilde{l}$ and $R^{-1}(\tilde{l})$ which is the distance $\Delta$ between $\tilde{l}$ and $\tilde{l}_{-1}$ along $U_{\eta}(\gamma)$.  
 
Again by Lemma \ref{stab0} there are two possibilities for the location of the endpoints of $\gamma$. In the first, $\gamma_+>\frac{1}{3}>R^{-3}(\infty)$ and $\gamma_-\in (-\infty,0).$ Then $U_{\eta}(\gamma_+)\in (-3,0)$ and $U_{\eta}(\gamma_-)>0$. Let $\alpha$ be the geodesic whose endpoints are -3 and Int$[U_{\eta}(\gamma_-)]+1$ where Int$[x]$ is the integral part of $x$.   Then $\Delta$ is bounded below by the distance between   $\tilde{l}$ and $\tilde{l}_{-1}$ along $\alpha$. By Lemma \ref{ab} this value is bounded below by $\delta(2,3).$

The other possibility is  $\gamma_-<-\frac{1}{3}<R^{3}(\infty)$ and $\gamma_+\in (0,1).$ Then $U_{\eta}(\gamma_+)<-1$ and $U_{\eta}(\gamma_-)\in (0,3)$. Let $\alpha$ be the geodesic whose endpoints are Int$[U_{\eta}(\gamma_+)]$ and 3. Then $\Delta$ is bounded below by the distance between   $\tilde{l}$ and $\tilde{l}_{-1}$ along $\alpha$. Again by Lemma \ref{ab} this value is bounded below by $\delta(3,4).$\qed \\
 
\end{prolist}

  \noindent {\bf Remarks.}
 It follows from the proposition that Theorems 4 and 6 of
\cite{haas} hold for all finite area non-compact orbifolds and the Fushsian groups representing them, in particular for Hecke groups.
% \eject

 \subsection{The measure of a thickened section}
 Let $\mathcal{A}^*(z)$ denote either  $\mathcal{A}(z)$ or $z$, depending on whether the star is $a$ or nothing. 
\begin{prop}\label{LL}
For $\epsilon <\delta(\Gamma)$, $k<3$ for excursions and $k\leq 2$ for approximating excursions
$$ \mu(\mathcal{L}_k^*(\epsilon))=\frac{\epsilon \, \mathcal{A^*}(k)}{\pi\, \text{area}(S) }.$$
\end{prop}

\begin{proof}
We prove  Proposition \ref{LL} for $\mathcal{L}_k^a(\epsilon)$. The computation for regular excursions is similar and easier.
As a result of Proposition \ref{thickL} the projection of $\tilde{\mathcal{L}}^a_k(\epsilon)$ to $S$ is injective. Therefore
$\mu(\mathcal{L}^a_k(\epsilon))=\tilde{\mu}(\tilde{\mathcal{L}}^a_k(\epsilon))$. 
This last quantity is $\epsilon\tilde{\nu}(\tilde{\mathcal{L}}_k).$ Set $E_k^a=\{(x,y)\,|\, (x,y,t)\in\mathcal{L}_k^a \} $.
$E^a_k$ is a full measure subset of $\Delta_k$ \cite{haas}.   Therefore
\beq\label{nu}
\tilde{\nu}(\tilde{\mathcal{L}}^a_k)=
\int_{E_k^a} d\tilde{\nu}(x)=
\frac{1}{\pi\, \text{area}(S) }\int_{\Delta_k}\frac{1}{(x-y)^2}dydx. 
\eeq
%\frac{2}{\pi \text{area}(S) }\int_{\Delta^-_k}\frac{1}{(x-y)^2}dydx= 
For $1\leq k\leq2$, line (\ref{nu}) becomes
\beq\label{nu2}
\frac{2}{\pi\, \text{area}(S) }\left (\int_0^{\frac{2}{k}-1}\int_{-\infty}^{x-\frac{2}{k}} \frac{1}{(x-y)^2}dydx+
\int_{\frac{2}{k}-1}^1\int_{-\infty}^{-1}\frac{1}{(x-y)^2}dydx\right )
\eeq
which gives  $\d{\mu(\mathcal{L}_k^a(\epsilon))=\frac{\epsilon}{\pi\, \text{area}(S) }(2-k+2\log k).}$

When $k<1$, the second integral in line (\ref{nu2}) disappears and the limits of the first integral  in the $x$ variable, go from 0 to 1. The result is  $\d{\mu(\mathcal{L}^a_k(\epsilon))=\frac{\epsilon}{\pi\, \text{area}(S) }k.}$
\end{proof}

\subsection{Proof of Theorem \ref{counting}}
 Fix $k\leq 2$ for approximating excursions and $k<2$ for excursions.  Suppose ${\mathbf v}\in \mathcal{E}$ and $\gamma=\gamma({\mathbf v}).$ Since the depth and  excursion parameters alternate, between 0 and $t>0$ there can lie at most two  depth parameters more than excursion parameters but at most two fewer depth parameters.   It follows that for ${\mathbf v}\in \mathcal{E} $ and $ \epsilon\leq\delta(\Gamma)$,
\begin{equation}\label{floweqn}
 \int_0^t\boldsymbol{\chi}_{\mathcal{L}^*_k(\epsilon)}\left (G_{\tau}({\mathbf v})\right )
d\tau-2\epsilon \leq \epsilon N^*_{\mathbf v}(k)(t) \leq
\int_0^t\boldsymbol{\chi}_{\mathcal{L}^*_k(\epsilon)}\left (G_{\tau}({\mathbf v})\right )
d\tau+2\epsilon.
\end{equation}

By the Ergodic Theorem for flows, for almost all ${\mathbf v}\in \mathcal{E}$

\begin{equation}\label{floweqn2}
\lim_{t\rightarrow\infty}\frac{1}{t}\int_0^t\boldsymbol{\chi}_{\mathcal{L}^*_k(\epsilon)}\left
(G_{\tau}({\mathbf v})\right ) d\tau= \int_{T_1S}\boldsymbol{\chi}_{\mathcal{L}^*_k(\epsilon)}({\mathbf v})
d\mu=\mu(\mathcal{L}^*_k(\epsilon))
\end{equation}
%\beq\label{floweqn3}
%=
%\frac{\epsilon}{\pi \text{area}(S) }
%\mathcal{A}^*(k).\,\,\,\,\,\,\,\,\,\,\,\,\,\,\,\,\,\,\,\,\,
%\eeq
Divide through by $t$ and $\epsilon$ in (\ref{floweqn}) and let  $t$ go to infinity.
Using (\ref{floweqn2}) and Proposition \ref{LL}   to compute the limit, we have shown that for fixed $k$, Theorem \ref{counting} hold for almost all ${\mathbf v}\in\mathcal{E}.$

In order to complete the proof of Theorem \ref{counting}, we need to show that  the limits  hold  simultaneously for all $ k\leq 2$ for approximating excursions and $k<2$ for excursions, on a set of full measure in $T_1S$. Using basic properties of measurable sets, we can assume that for  all ${\mathbf v}$ in a full measure subset $\mathcal{R}\subset \mathcal{E} $, Theorem \ref{counting} holds for a countable dense subset $\mathcal{D}$ of values $k$ in $(0,2]$ or $(0,2).$ Without loss of generality suppose $2\in \mathcal{D}$ when dealing with approximating excursions. Then for  any $k<2$, there exist sequences of arbitrary small numbers $\delta_n$ and $\zeta_n$ so that for any $n\in \NN$, $k-\delta_n\in \mathcal{D}$ and $k+\zeta_n\in \mathcal{D}.$
 Then for ${\mathbf v}\in\mathcal{R}$
 \beqq
\frac{\mathcal{A}^*(k-\delta_n)}{\pi\, \text{area}(S) }= \lim_{t\rightarrow\infty}\frac{ N^*_{\mathbf v}(k-\delta_n)(t)}{t}\leq \liminf_{t\rightarrow\infty}\frac{ N^*_{\mathbf v}(k )(t)}{t}\leq
\eeqq
\beqq
 \limsup_{t\rightarrow\infty}\frac{ N^*_{\mathbf v}(k )(t)}{t}\leq\lim_{t\rightarrow\infty}\frac{ N^*_{\mathbf v}(k+\zeta_n)(t)}{t}=
\frac{ \mathcal{A}^*(k+\zeta_n)}{\pi\, \text{area}(S) }
  \eeqq
Letting  $n\rightarrow\infty$ we see that   Theorem \ref{counting} holds for all appropriate $k$.
\qed\\
 
 \noindent{\bf Remark.} In proving Theorem \ref{counting} for $k$-excursions, it is possible to get a  stronger result which includes values of $2\leq k<3$. When $k<2$ there is a one-to-one correspondence between excursions, excursion parameters and depth paremeters. But when $k\geq 2$ there exist excursions which do not have corresponding excursion parameters. The integrals in (\ref{floweqn}) count excursion parameters. In light of this we extend the original definition of  $N_{\mathbf v}(k)(t)$, for values of $k\geq 2$, to be  the number of depth parameters $t_i\leq t$ for which there is a corresponding excursion parameter. The above argument gives a proof of of Theorem \ref{counting}  with this extended definition for $k\geq 2$. This is exactly what is needed to prove Proposition \ref{prop} below.   
 
\subsection{Proof of Corollary \ref{return}}
Let ${\mathbf v}$ be a unit tangent vector for which the conclusions of Theorem \ref{counting} hold and let $t_n$ be the depth parameters for the $*$-excursions  along $\gamma=\gamma({\mathbf v})$. Then 
$N^*_{\mathbf v}(k)(t_n)=n$ and the average length of an excursion is
\beqq
\lim_{n\rightarrow\infty}\frac{1}{n}\sum_{i=1}^{n-1}t_{i+1}-t_i  \,\,=\,\,\lim_{n\rightarrow\infty}\frac{t_n}{n}\,\,=\,\,\lim_{n\rightarrow\infty}\frac{t_n}{N^*_{\mathbf v}(k)(t_n)}\,\,=\,\,\frac{\pi\, area(S)}{\mathcal{A}^*(k)}
\eeqq
as asserted in Part (i).

Let $\delta_n$ be the boundary parameter corresponding to  the $n^{th}$ intersection of $\gamma$ with the boundary, $\partial C_k$. Since $k\leq 1$ the interiors of horocycles in $\HH$ covering $C_k$ are all disjoint. It follows that $\delta_{2i-1}< t_i<\delta_{2i}$ (this may be off by one if $\gamma(0)\not\in C_k$.) Then the average length of an excursion is 
\beq\label{cusptime}
\lim_{n\rightarrow\infty}\frac{1}{n}\sum_{i=1}^{n}\delta_{2i}-\delta_{2i-1}=  
  \lim_{n\rightarrow\infty}\frac{\int_0^{t_n}\chi_{(T_1C_k)}(G_t({\mathbf v}))\,dt}{N_{\mathbf v}(k)(t_n)}=
  \eeq
\beqq
\lim_{n\rightarrow\infty}\left (\frac{1}{t_n}\int_0^{t_n}\chi_{(T_1C_k)}(G_t({\mathbf v}))\,dt\times \frac{t_n}{N_{\mathbf v}(k)(t_n)}\right )=
\eeqq
\beq\label{flow}
   \lim_{n\rightarrow\infty}\frac{1}{t_n}\int_0^{t_n}\chi_{(T_1C_k)}(G_t({\mathbf v}))dt \times  \lim_{n\rightarrow\infty}\frac{t_n}{N_{\mathbf v}(k)(t_n)}
   % \,\,=\,\,  \frac{k}{area(S)} \times \frac{\pi\, area(S)}{k}
    \,\,=\,\,\pi ,
\eeq
where we have  further stipulated that ${\mathbf v}$ comes from the full measure set on which the first limit in line (\ref{flow}) above converges. By the Ergodic Theorem for flows, for almost all ${\mathbf v}$, the first limit  in (\ref{flow}) will converge to $k/\text{area}(S)$. The value of the second limit comes from Theorem \ref{counting}. \qed

\subsection{Some special vectors and geodesics}\label{special}
This is a good time to introduce a special set of unit tangent vectors that will be useful when we turn our attention to rational approximation.  Given $x\in (0,1),$ let $\tilde{{\mathbf v}}_x$ be the unit tangent vector in $T_1\HH$ based at the point $x+2i$ pointing in the direction of $x$, let ${\mathbf v}_x=\pi_*(\tilde{{\mathbf v}}_x)$, let $\tilde{\beta}_x=\tilde{\gamma}(\tilde{{\mathbf v}}_x)$ and let $\beta_x=\gamma({\mathbf v}_x)$. The geodesic $\tilde{\beta}_x$ is a vertical Euclidean ray in $\HH$ with the endpoint $x$. 

\begin{prop}\label{prop}
For almost all $x\in (0,1)$ and $k\in (0,3)$ the limits (\ref{overlap})  in Theorem \ref{counting} hold for ${\mathbf v}={\mathbf v}_x.$
\end{prop} 

\begin{proof}
For almost all $x\in (0,1)$ there is a vector ${\mathbf v}\in \mathcal{E}$ for which the limit (\ref{overlap}) of Theorem \ref{counting} holds, so that the geodesic $\gamma$ has a lift $\tilde{\gamma}$ with $\tilde{\gamma}_+=x.$
 The geodesics $\tilde{\beta}_x$  and  $\tilde{\gamma}$ are asymptotic and therefore, given $\epsilon >0$, beyond some point on each geodesic,   the geodesics are within $\epsilon$ of one another. It follows that for all $k\in (0, 3)$ and $\epsilon$ sufficiently small.
\beqq
\frac{ k-\epsilon }{\pi\, \text{area}(S) } =\lim_{t\rightarrow\infty}\frac{ N_{\mathbf v}(k-\epsilon)(t)}{t}\leq \liminf_{t\rightarrow\infty}\frac{ N_{{\mathbf v}_x}(k )(t)}{t}\leq
\eeqq
\beqq
 \limsup_{t\rightarrow\infty}\frac{ N_{{\mathbf v}_x}(k )(t)}{t}\leq\lim_{t\rightarrow\infty}\frac{ N_{\mathbf v}(k+\epsilon)(t)}{t}=\frac{ k+\epsilon }{\pi\, \text{area}(S) }  .
 \eeqq
Since this holds for all $\epsilon>0$,  limits (\ref{overlap})  hold for ${\mathbf v}={\mathbf v}_x.$  
 \end{proof} 
 
\noindent{\bf Remark.} There is a version of the proposition for approximating excursions that will be proved in Section \ref{behave}.

\section{The Sequence of Excursion Depths}\label{xdepth}
\subsection{The discrete logarithm law}\label{sectionlog}

We shall define values $D({\mathbf v})(n)$ and  $D^a({\mathbf v})(n)$ that  quantify the depth of the $ n^{th}$  $2$-excursion, or approximating $2$-excursion, along the geodesic $\gamma({\mathbf v})$.  Although it is possible to do the analysis, there is   little to be gained working with arbitrary $k<2$.  

 If $\gamma$ is tangent  to the   neighborhood $C_{d_n}$ of the cusp $P$ at the point $\gamma(t_n)$, then we define $D({\mathbf v})(n)=d_n$.  Also, set $D^a({\mathbf v})(j)=D({\mathbf v})(n_j) $, where  $\gamma(t_{n_j})$ is the $j^{th}$ approximating $2$-excursion along $\gamma$.   These are defined for almost all ${\mathbf v}\in T_1S.$
%In order to simplify things later, we rephrase the above as follows: let  $d(\gamma(\tau_n))=D^*({\mathbf v})(n)$ where the $\tau_n$ are  depth parameters for the *2-excursions along $\gamma.$ (Here * denotes the word approximating or nothing.)

\begin{thm}\label{2thm}
For almost all ${\mathbf v}\in T_1S$
 \begin{equation}\label{discsullivan}
%\limsup_{n\rightarrow\infty}\frac{\log D(v)(n)}{\log n}= 
\limsup_{n\rightarrow\infty}-\frac{\log D^*({\mathbf v})(n)}{\log n} =1.
\end{equation}
\end{thm}

This gives a discretization of the logarithm law \cite{sullivan} in dimension two. We see in Section \ref{0thmtheta} how Theorem \ref{theta} follows from the above.

Let $\text{dist}({\mathbf v})(t)$
denote the distance on $S$ between the points $\gamma({\mathbf v})(0)$ and $\gamma({\mathbf v})(t)$ and let $d(\gamma(t))$ be the area of the largest cusp neighborhood that does not contain $\gamma(t).$ 

\begin{lemma}
Suppose   $\{t_{n}\}$ is a sequence of numbers so that $\gamma(t_{n})\in C_2$ for all $n\in \NN$ and  at least one of the sequences
$\{\text{dist}({\mathbf v})(t_{n})\}$  or $\{-\log d(\gamma(t_{n}))\} $ diverges to infinity, 
then
\beq\label{limit1}
\lim_{n\rightarrow\infty}-\frac{\text{dist}({\mathbf v})(t_{n})}{\log d(\gamma(t_{n}))}=1.
\eeq
\end{lemma}

\begin{proof}
Let $ \alpha:[0,\gamma(t_{n})] \rightarrow S$ be the minimal length geodesic from $
\alpha(0)=\gamma(0)$ to $\gamma(t_{n})$. Suppose $d(\gamma (t_{n}))  =k_n$.
Choose $s>1$ so that $\gamma(0)\not\in C_s$. Let $ a$ denote the minimal distance from  $\gamma(0)$ to the boundary of $ C_s$.  By an easy computation in $\HH$, the distance from the boundary of $C_s$ to $C_{k_n}$ is $\log s-\log k_n$, while the distance around the boundary of $C_k$ for $k>1$ is $k$. Consider  the piecewise geodesic $\beta_n$ that take the shortest path from $\alpha(0)$ to the boundary of $C_s$, around the boundary of $C_s$ to the start of the minimal length geodesic arc from the boundary of $C_s$ to $C_{k_n}$ and around the boundary of $C_{k_n}$ to $\alpha (r_{n})$.  Length estimates using the triangle inequality  show that
\beqq
a+  \log s-\log k_n<\text{dist}({\mathbf v})(t_{n})< a+ s+\log s - \log k_n+k_n.
\eeqq

It is clear from the equation that $\{\text{dist}({\mathbf v})(t_{n})\}$  diverges to infinity if and only if $\{-\log k_n\} $ diverges to infinity.
Since $k_n\rightarrow 0$,   while the remaining values are bounded, we have
\beqq
\lim_{n\rightarrow\infty} \frac{a-\log k_{n}+\log s    }{a+ s+ \log s - \log k_n +k_n  }=1.
\eeqq
\end{proof}

\noindent {\em Proof of Theorem \ref{2thm}.}
 Let  $d(\gamma(\tau_i))=D^*({\mathbf v})(i)$ where the $\tau_i$ are  depth parameters for the *2-excursions along $\gamma.$ (Here * denotes the word approximating or nothing.)
 Then Theorem \ref{counting} takes the form
\beqq 
 \lim_{i\rightarrow\infty}\frac{N^*_{\mathbf v}(k)(\tau_i)}{\tau_i} =\lim_{i\rightarrow\infty}\frac{i}{\tau_i}= \frac{ \mathcal{A}^*(k) }{\pi\, \text{area}(S)}.
\eeqq
In particular, this gives
\beq\label{loglog}
\lim_{i\rightarrow\infty}\frac{\log i}{\log \tau_i}=1
\eeq
It follows from the proof of the logarithm law in \cite{sullivan},   that for almost all ${\mathbf v}\in T_1S$
\beq\label{sullivan2}
 \limsup_{t\rightarrow\infty}\frac{\text{dist}({\mathbf v})(t)}{\log
t}=  \limsup_{i\rightarrow\infty}\frac{\text{dist}({\mathbf v})(\tau_{i})}{ \log \tau_{i} }=1. 
 \eeq    
Consequently, $\text{dist}({\mathbf v})(\tau_{i}) \rightarrow\infty$ and by the previous lemma 
\beq\label{dog}
\lim_{i\rightarrow \infty} -\frac{\log(d(\gamma(\tau_i))}{\text{dist}({\mathbf v})(\tau_{i})}=1.
\eeq
   For the remainder of the argument we suppose that ${\mathbf v}\in \mathcal{E}_s\subset \mathcal{E}$ where   $\mathcal{E}_s$ is a set of full measure  for which the limit (\ref{sullivan2}) holds.     
 
Employing all of the limits (\ref{loglog}), (\ref{sullivan2}) and (\ref{dog}), we have
\beqq
%\label{num} 
\limsup_{i\rightarrow\infty} -\frac{\log(D^*({\mathbf v})(i))}{\log i}=
\limsup_{i\rightarrow\infty} -\frac{\log(d(\gamma(\tau_i)))}{\log i}=
\eeqq 
\beqq
\limsup_{i\rightarrow\infty}\left (-\frac{\log(d(\gamma(\tau_i))}{\text{dist}({\mathbf v})(\tau_{i})}\times\, 
\frac{\text{dist}({\mathbf v})(\tau_{i})}{\log \tau_i}\times\, \frac{\log \tau_i}{\log i}\right )=
\limsup_{n\rightarrow\infty}\frac{\text{dist}({\mathbf v})(\tau_{i})}{ \log \tau_{i} }=1. 
 \eeqq   \qed 
 \\
\noindent {\bf Remark.} The limit (\ref{loglog}), which follows from Theorem \ref{counting} and is used above, can also be easily derived from the estimates  in \cite{strat}. 
 
\subsection{Some interesting averages}
Using Theorem \ref{counting} we derive four averages involving the values $D^*({\mathbf v})(n)$.   They will later be translated into Theorem \ref{thetasums}.

  \begin{thm}\label{before}
  For almost all ${\mathbf v}\in T_1S$  
\beqq\label{nothing}
\lim_{ n \rightarrow\infty}\frac{1}{n}\sum_{i=1}^n  D^*({\mathbf v})(i)=
\left\{ \begin{array}{ll}
1 & \mbox{if\, *=nothing}\\
1/(2\log 2)  & \mbox{if *=a }
\end{array}
\right.
\eeqq
%\end{thm}
and
%\begin{thm}For almost all $v\in T_1S$ and for all $k\leq 2$
\beqq\label{a}
\lim_{ n \rightarrow\infty}\frac{1}{n}\sum_{i=1}^n  \log D^*({\mathbf v})(i)= 
\left\{ \begin{array}{ll}
\log 2- 1 & \mbox{if\, *=nothing}\\
\frac{1}{2}\log 2 -1  & \mbox{if *=a }
\end{array}
\right.
\eeqq
\end{thm}
  
\begin{proof} 
For $x\in (0,2]$ define the function
\beqq
F_{\mathbf v}^*(x)=
 \left\{ \begin{array}{ll}
x/2 & \mbox{if\, *=nothing}\\
 \mathcal{A}(x)/2\log 2  & \mbox{if *=a }
\end{array}
\right.
\eeqq 
By Theorem \ref{counting}, for almost all ${\mathbf v}\in T_1S$
\beqq
F_{\mathbf v}^*(x)=\lim_{n\rightarrow\infty}\frac{N_{\mathbf v}^*(x)(n)}{ N_{\mathbf v}^*(2)(n)}=\lim_{n\rightarrow\infty}
\frac{1}{n}\#\{  0<i\leq n \, |\, D^*({\mathbf v})(i)\leq x\}
\eeqq
 is the limiting distribution of the arithmetic function $D^*({\mathbf v}):\NN\rightarrow [0,2].$
 The first set of averages in Theorem \ref{before} then follow easily from Theorem \ref{counting}, by computing the expectation $\int_0^2x\,dF_{\mathbf v}^*(x).$
 
 In order to compute the second set of averages, observe that
 \beqq
  H^*_{\mathbf v}(x)=\lim_{n\rightarrow\infty}
\frac{1}{n}\#\{  0<i\leq n \, |\, \log D^*({\mathbf v})(i)\leq x\}
\eeqq
 converges for almost all ${\mathbf v}\in T_1S$. In fact, $H^*_{\mathbf v}(x)=F^*_{\mathbf v}(e^x).$ The second set of averages again follows by computing  $\int_{-\infty}^{\log 2}x\,dH_{\mathbf v}^*(x).$ 
 
 \end{proof}

\section{Diophantine Approximation with Respect to $\Gamma$}\label{approximation}
\subsection{$\Gamma$-rational numbers}

The orbit of $\infty$ under the action of $\Gamma$ is called the $\Gamma$-rational numbers  and is written
$\QQ(\Gamma).$ For $g\in \Gamma$ we have
\beq
g(z)=\pm\pmtx{p & r \cr q & s}(z) =\frac{pz+r}{qz+s}
\eeq
where the matrix belongs to $\text{SL}_2(\RR).$
Thus $g(\infty)= p/q$ and if $g(\infty)\geq 0$ the sign can be chosen so that $p, q\geq 0.$ It is easily shown that, for a non-negative element of $\QQ(\Gamma)$ the numbers $p$ and $q$ are uniquely determined \cite{haas}.  Thus each point of $(0,1)$ in the $\Gamma$-orbit of infinity has a well defined representation as a\lq\lq fraction" $p/q.$

We shall suppose that $P$ is the cusp of $S$ corresponding to the point $\infty\in\partial\HH.$  
Through the rest of this section we set $k=2$ and refer to an (approximating)  $2$-excursion simply as an (approximating) excursion. A lift $\tilde{\gamma}$ of $\gamma$ and an   excursion $e=\gamma(t_e)$ along $\gamma$ determine  a $\Gamma$-rational number  $p/q=\varphi(\tilde{\gamma}, e)$ as follows. Let $\sigma$ be the geodesic ray orthogonal to $\gamma$ at $e,$ heading directly out to infinity in the cusp neighborhood $C_2$ of  $P$ and let $\tilde{\sigma}$ denote the lift of $\sigma$ to $\HH$ so that  $\tilde{\sigma}(0)=\tilde{\gamma}(t_e)=\tilde{e}.$ Then set 
\beq
\varphi(\tilde{\gamma}, e)=\frac{p}{q}=\lim_{t\rightarrow\infty}\tilde{\sigma}(t).
\eeq

 Recall the definition of $\beta_x$ from Section \ref{special}. Given $x\in (0,1),$     $p/q\in \QQ(\Gamma)$ is said to be a  $\Gamma$-{\em convergent} of $x$,  if  $p/q=\varphi(\tilde{\beta}_x, a)$ for some   approximating excursion $a$ of $\beta_x.$ If  $p/q=\varphi(\tilde{\beta}_x, e)$ for an excursion $e$ that is not an approximating excursion, then $p/q$ is called a  {\em non-classical} $\Gamma$-convergent  approximating $x$. The sequence consisting of the $\Gamma$-convergents and the   non-classical $\Gamma$-convergents  will be called the sequence of  $\Gamma$-{\em n-convergents} of $x$.  If the number $x\in (0,1)$ has infinitely many distinct $\Gamma$-n-convergents  then we call $x$ a $\Gamma$-{\it irrational number}.  

 \subsection{The geometry of $\Gamma$-convergents}
 \begin{lemma}\label{frac}
Given $x\in (0,1)$,  $p/q\in \QQ(\Gamma)$ is a $\Gamma$-convergent to $x$ if and only if 
  $\tilde{\beta}_x$ intersects two or more geodesics in $\mathscr{L}$ with the endpoint $p/q.$

\end{lemma}

\begin{proof}
Consider the  lift $\tilde{\gamma}_{a}$ of the geodesic 
$\gamma$  that  is associated with an approximating  excursion $a$. It is tangent to the half-plane $H_{\frac{1}{d(a)}}$ at the point $\tilde{\gamma}_{a}(t_a)=\tilde{a}.$ Thus the lift of $\sigma$ beginning at $\tilde{a} ,$ is a vertical ray 
from $\tilde{a}$ to infinity. Since $a$ is an approximating excursion, $\tilde{\gamma}_{a}$ has its endpoints in $I$ and intersects $\tilde{l}$ and at least one  of the other vertical rays with integral real part. It follows that at least two geodesics in $\mathscr{L}$ share the endpoint $p/q$ with $\tilde{\sigma}$ and intersect the geodesic ray $\tilde{\gamma}$. Conversely, if the ray $\tilde{\gamma}$ intersects two or more geodesics in $\mathscr{L}$ which share an endpoint, then that endpoint is $\varphi(\tilde{\gamma},a)$ for some approximating excursion along $\gamma$. 

 \end{proof}

\noindent{\bf Remark.} Using our definitions it is possible to get 0 as a convergent to a number $x$.   If $\Gamma=\text{PSL}_2(\ZZ)$, this will happen for irrational numbers $x=[a_1,a_2,\ldots]$ where $a_1\not= 1$ because $\tilde{\beta}_x$ will then intersect at least two geodesics in $\mathscr{L}$ with the endpoint 0. Otherwise, the arguments in  \cite{haas} show that the   $\Gamma$-convergents to $x$ are precisely the convergents coming from the partial quotients of the continued fraction expansion for $x$. This fact motivates the definitions we have used for $\Gamma$-convergents and the approximating excursions used to define them. Also, in this setting the non-classical $\Gamma$-convergents correspond to a subset of the nearest mediants of $x,$ 
\cite{bosma}.
 
\subsection{Almost all ${\mathbf v}_x$ are well behaved}\label{behave}

Let $\mathcal{E}^+\subset T_1S$ be the full measure subset so that for ${\mathbf v}\in \mathcal{E}^+$  all of the conclusions  of Theorems \ref{counting} and \ref{2thm} hold. Recall the definition of ${\mathbf v}_x$ from Section \ref{special}.

\begin{thm}\label{next}
For almost all $x\in (0,1),\,  {\mathbf v}_x\in  \mathcal{E}^+.$

\end{thm}

This theorem provides a means to translate the geometric results like Theorem \ref{counting}, Theorem \ref{2thm} and Theorem \ref{before} into their number theoretic analogs.

 \begin{proof}

The part of the theorem dealing with  formula (\ref{overlap}) from Theorem \ref{counting} was already proved in Proposition \ref{prop}. We shall first address Theorem \ref{2thm} for approximating excursions, then Theorem \ref{counting} for approximating excursions and then Theorem \ref{2thm} for excursions.

Let $\tilde{\mathcal{E}}^+$ denote the lifts of vectors in $\mathcal{E}^+$. Notice that if $(\zeta,\psi,t)\in\tilde{\mathcal{E}}^+$ for one value of $t$ then it is in there for all values of $t$. By Proposition \ref{thickL} (ii), for $x$ in a full measure set ${\bf B}\subset (0,1)$, almost all  $y\not\in [0,1]$ and $t\in \RR$, $\tilde{{\mathbf v}}=(x,y,t)\in  \tilde{\mathcal{E}}^+$. We choose $t$ so that $\tilde{\gamma}=\tilde{\gamma}(\tilde{{\mathbf v}})$ satisfies $0<\text{Re}(\tilde{\gamma}(0))<1$ and $| \tilde{\gamma}(0)-1/2|>0.$ Suppose from here on that $x\in {\bf B}.$

Let $\mathscr{L}_0$ denote the set of geodesics in $\mathscr{L}$ with both of their endpoints in $[0,1].$  Given the normalization of  $\tilde{{\mathbf v}}$ and the fact that the geodesics $\tilde{\beta}_x$ and $\tilde{\gamma}$ are asymptotic at $x$, $\tilde{\beta}_x$  will intersect  a geodesic in $ \mathscr{L}_0$   if and only if $\tilde{\gamma}$ intersects that same geodesic. In particular,  $\tilde{\beta}_x$ intersects a set of geodesics in $ \mathscr{L}_0$ that share an endpoint  if and only if $\tilde{\gamma}$ intersects that same set of geodesics. Using Lemma \ref{frac}, this shows that there is a one-to-one correspondence between the set $\{a_i\}$ of approximating excursions along the ray $\beta_x$ and the set $\{e_i\}$ of approximating excursions along the ray $\gamma$, so that 
$\varphi(\tilde{\beta}_x,a_i)=\varphi(\tilde{\gamma}, e_i).$

A further consequence of the geodesics being asymptotic is that given $\epsilon >0$, by making $t, \tau$ sufficiently large, every point at which one of $ {\gamma}(t)$ or $ {\beta}_x(\tau)$ is tangent to a cusp neighborhood of $P$ is within $\epsilon$ of a point at which the other geodesic is tangent to a cusp neighborhood of $P$. More precisely, given $\epsilon>0$ there exists an integer $K>0$ so that for $i>K,\, \text{dist}( e_i, a_i)<\epsilon.$ Using the hyperbolic metric in the upper half-plane we can get the estimate: $|\log(\text{Im}(\tilde{e}_{i}))-\log(\text{Im}(\tilde{a}_i))|<\epsilon.$

Since Im$(\tilde{e}_{i})=1/d(e_i)$ and Im$(\tilde{a}_{i})=1/d(a_i),$ we have
$|\log(d(a_i))-\log(d(e_i))|<\epsilon$ or equivalently, $|\log D^a({\mathbf v}_x)(i)-\log D^a({\mathbf v})(i)|<\epsilon.$ It follows that 
\beqq
\limsup_{i\rightarrow\infty}\frac{ \log D^a({\mathbf v}_x)(i)}{\log i}=
\limsup_{i\rightarrow\infty}\frac{ \log D^a({\mathbf v})(i)}{\log i}=1,
\eeqq
completing the first part of the argument.

Now suppose that the approximating excursions along $\tilde{\gamma}$ and $\tilde{\beta}_x$ have excursion parameters $t_i$ and $\tau_i$ respectively; that is,  $\tilde{\gamma}(t_i)=e_i$ and $\tilde{\beta}_x(\tau_i)=a_i.$  Since $\text{dist}(e_i,a_i)\rightarrow 0$, there exists a constant $m>0$ so that
\beq\label{tau}
\tau_i-m<t_i<\tau_i+m
\eeq
By Proposition \ref{thickL} there is a lower bound $b$ for the distance between consecutive excursion parameters for approximating excursions along a geodesic. 
Putting these observations together we get:
\beqq
\frac{N^a_{\mathbf v}(k)(t_i)-\frac{m}{b}}{t_i}\leq 
\frac{N^a_{\mathbf v}(k)(t_i-m)}{t_i}\leq 
\frac{N^a_{{\mathbf v}_x}(k)(\tau_i)}{t_i}\leq 
\eeqq
\beqq
\frac{N^a_{\mathbf v}(k)(t_i+m)}{t_i}\leq
\frac{N^a_{\mathbf v}(k)(t_i)+\frac{m}{b}}{t_i}.
\eeqq
It follows that 
\beqq
\lim_{i\rightarrow\infty}\frac{N^a_{\mathbf v}(k)(t_i)}{t_i} = \lim_{i\rightarrow\infty}\frac{N^a_{\mathbf v_x}(k)(\tau_i)}{t_i}.
\eeqq
Using (\ref{tau}) a second time, the above implies
\beqq
\lim_{i\rightarrow\infty}\frac{N^a_{v_x}(k)(\tau_i)}{\tau_i}=\mathcal{A}(k)
\eeqq
 and it follows that formula (\ref{app}) holds for the vector ${\mathbf v}_x$.
%\beqq
%\lim_{t\rightarrow\infty}\frac{N^a_{v_x}(k)(t)}{t}=\mathcal{A}(k).
%\eeqq

Finally, we need to prove that the conclusion of Theorem \ref{2thm} for excursions holds for ${\mathbf v}_x$. Let $\{t_i\}$ be the set of excursion parameters for excursions along $\tilde{\beta}_x$ and let $\{t_{i_j}\}$ be the subset of parameters for approximating excursions.

Observe that 
\beqq
\frac{1}{\log 2}= \lim_{t\rightarrow\infty}\frac{N_{{\mathbf v}_x}(2)(t)}{t}\times\lim_{t\rightarrow\infty}\frac{t}{N^a_{{\mathbf v}_x}(2)(t)}=
%\eeqq
%\beqq
\lim_{j\rightarrow\infty}\frac{i_j}{t_{i_j}}\times\lim_{j\rightarrow\infty}\frac{{t_{i_j}}}{j}=\lim_{j\rightarrow\infty}\frac{i_j}{j}.
\eeqq

Then we have
\beqq
\limsup_{i\rightarrow\infty}-\frac{ \log D({\mathbf v}_x)(i)}{\log i}=
\limsup_{i\rightarrow\infty}-\frac{\log (d(\tilde{\beta}_x(t_i)))}{\log  i }
\eeqq
but since all $k$-excursions with $k<1$ are approximating excursions, this is equal to 
\beqq
\limsup_{j\rightarrow\infty}-\frac{\log (d(\tilde{\beta}_x(t_{i_j})))}{\log (i_j)}=
\limsup_{j\rightarrow\infty}\left (-\frac{\log (d(\tilde{\beta}_x( i_j )))}{\log(i_j)}\times
\frac{\log (i_j)}{\log j}\right ).
\eeqq
By the previous observation and the definition of $D^a$ this equals
\beqq
\limsup_{j\rightarrow\infty}-\frac{ \log D^a({\mathbf v}_x)(j)}{\log j} 
\eeqq
which has already been shown to be 1.
\end{proof}

 \subsection{The proofs of Theorems \ref{0thm}, \ref{theta} and \ref{thetasums}}\label{0thmtheta}
 
  Given $ \frac{p}{q}\in\QQ(\Gamma)$, the horocycle $\mathcal{H}_{\frac{p}{q}}(s)$ is defined to be the Euclidean disc of radius $\frac{s}{q^2}$ in
$\HH$, tangent to $\RR$ at the point $\frac{p}{q}.$ Since $\frac{p}{q}\in \QQ(\Gamma),$ there exists a transformation $g\in\Gamma$ with $g(\frac{p}{q})=\infty.$ By an easy calculation \cite{haas}, $g(\mathcal{H}_{\frac{p}{q}} (\frac{r}{2}))=H_{\frac{1}{r}}=\{z\,|\,\text{Im}(z)>\frac{1}{r}\}$ and thus $\mathcal{H}_{\frac{p}{q}}(\frac{r}{2}  )$ is a lift of the cusp neighborhood $C_r$. Let $e=\gamma(t_e)$ be an excursion or an approximating excursion of $\gamma$. If $\tilde{\gamma}$ is a lift of $\gamma$ then for $\frac{p}{q}=\varphi(\tilde{\gamma},e),\, \tilde{\gamma}$ is tangent to $\mathcal{H}_{\frac{p}{q}}(  \frac{d(e)}{2} )$ at $\tilde{\gamma}(t_e).$

Henceforth we assume that $x\in (0,1)$ with $v_x\in\mathcal{E}^+$ and $\{e_i=\ \beta_x(t_j)\}$ is the sequence of either excursions or approximating excursions along $\beta_x$. Let $\{\frac{p_j}{q_j}=\varphi(\tilde{\beta}_x, e_i)\}$ be the associated sequence of convergents or n-convergents.  Note that since $\sigma_i$ is orthogonal to $ \tilde{\beta}_x$
\beq
\lim_{i\rightarrow\infty}\frac{p_i}{q_i}=
\lim_{i\rightarrow\infty}\tilde{\beta}_x(t_i)=
\lim_{t\rightarrow\infty}\tilde{\beta}_x(t)=x.
\eeq
In other words, the sequence  $\frac{p_i}{q_i}=\varphi(\tilde{\gamma}, e_i)$ does converge to $x$.
 Recall that $(\tilde{\beta}_x)_{e_j}$ is the special lift of $\beta_x$, which is tangent to $H_{\frac{1}{d(e_j)}}$ at the point $ (\tilde{\beta}_x)_{e_j}(t_{e_j})$. Let $g_j\in \Gamma$ be the transformation with $g_j(( \tilde{\beta}_x)_{e_j})=\tilde{
\beta}_x$ and $g_j(\infty)=\frac{p_j}{q_j}.$ Thus by the previous paragraph, $\tilde{e}_j= \tilde{\beta}_x(t_j)$ is the point at which $\tilde{\beta}_x$ is tangent to $\mathcal{H}_{\frac{p_j}{q_j}}(\frac{d(e_j)}{2}  )$. This is precisely the point $\tilde{e}_j=x+i\frac{d(e_j)}{2q_j^2}\in \HH.$

It follows from the above that the Euclidean distance from the center of $\mathscr{H}_{\frac{p_j}{q_j}}(\frac{d(e_j)}{2}  )$, which is 
 $\frac{p_j}{q_j}+i\frac{d(e_j)}{2q_j^2}$, to the point $\tilde{e}_j$ is equal to the distance form $\frac{p_j}{q_j}$ to $x$ on $\RR.$ But this is also the radius of the horocycle.  In other words,
 \beq\label{equ}
 |x-\frac{p_j}{q_j}|=\frac{d(e_j)}{2q_j^2}.
 \eeq
Since $d(e_j)=D^*({\mathbf v}_x)(j)$ we have proved that
%\begin{lemma}

 \beq\label{t*}
 \theta_j^*(x)=\frac{1}{2}D^*({\mathbf v}_x)(j).
 \eeq
% \end{lemma}
 Using this identity,  Theorems \ref{theta} and \ref{thetasums} follow immediately from Theorems \ref{2thm} and \ref{before} of Section \ref{xdepth}.    
 
We now turn to the proof of Theorem \ref{0thm}. Observe that $t_j-t_{j-1}$,  the distance along $\beta_x$ between consecutive excursions is 
\beqq
\text{dist}(\tilde{e}_j,\tilde{e}_{j-1})= \text{dist}\left (x+i\frac{d(e_j)}{2q_j^2},x+i\frac{d(e_{j-1})}{2q_{j-1}^2}\right )
\eeqq
computed in $\HH$. This value is easy to calculate and is
\beqq
\log   \frac{d(e_{j-1}) }{2q_{j-1}^2}-\log  \frac{d(e_{j}) }{2q_j^2}   =
\log |x-\frac{p_{j-1}}{q_{j-1}}| -\log |x-\frac{p_j}{q_j}|,
\eeqq
where the equality follows from (\ref{equ}).
 Then the average length of an excursion along $\beta_x$ is
 \beqq
\frac{\pi\, area(S)}{\mathcal{A}^*(2)}= \lim_{t\rightarrow\infty}\frac{t}{N^*_{{\mathbf v}_x}(2)(t)}=\lim_{n\rightarrow\infty}\frac{t_n}{n}=
 \lim_{n\rightarrow\infty} \frac{1}{n}\sum_{i=1}^n(t_i-t_{i-1})=
 \eeqq
 \beqq
  \lim_{n\rightarrow\infty} \frac{1}{n} \sum_{i=1}^n(\log |x-\frac{p_{j-1}}{q_{j-1}}| -\log |x-\frac{p_j}{q_j}| ) =
 \lim_{n\rightarrow\infty}- \frac{1}{n}\log |x-\frac{p_n}{q_n}|.
 \eeqq 
 This proves that the last limit exists and takes the value asserted in Theorem \ref{0thm}.
 
Then from  Theorem \ref{theta} and the above we can conclude that
 \beqq
 0=\lim_{n\rightarrow\infty}\frac{\log \theta^*_n(x)}{n}=\lim_{n\rightarrow\infty}\left (
 \frac{\log q_n^2}{n}+ \frac{\log |x-\frac{p_n}{q_n}|}{n}\right ),
 \eeqq
 which completes the proof of Theorem \ref{0thm}.\qed\\
 
Finally, we should mention that it follows easily from the above that
\beqq\label{nc}
|x-\frac{p}{q}|=\frac{\theta}{q^2}\,\, \text{for}\,\, \theta< 1
\eeqq
if and only if $\frac{p}{q}$ is  a $\Gamma$-n-convergent.

\subsection{Proofs of Corollary \ref{classical}}\label{cor pf}

\subsubsection{As a corollary to Theorem \ref{theta}} 
 
The irrational number $x\in\RR$ has the continued fraction expansion $x=a_0+[a_1,a_2,\ldots]$ and the rational approximants to $x\in(0,1)$ are given by $\frac{p_n}{q_n}=[a_1,\ldots,a_n] $. Define $a_n'=a_n+[a_{n+1},\ldots].$
From \cite{HW} we have
$$|x-\frac{p_n}{q_n}|=\frac{1}{q_n(a_{n+1}'q_n+a_{n-1})}$$
which gives
$$\theta_n(x)=\frac{1}{a_{n+1}'+\frac{q_{n-1}}{q_n}}.$$
As the $q_n$ are increasing and $a_{n+1}<a_{n+1}'<a_{n+1}+1$ we get
$a_{n+1}+1 < 1/\theta_n(x)< a_{n+1}+2.$
Thus we have
$$1=\limsup_{n\rightarrow\infty}-\frac{\log\theta_n(x)}{\log n} =\limsup_{n\rightarrow\infty}
\frac{\log a_{n+1}}{\log n}.$$\qed

\subsubsection{Classical proof}
The referee has pointed out a simple alternative to our proof of   Corollary \ref{classical}  based on purely classical methods.

 The result is a consequence of a theorem of Borel and F. Bernstein, see \cite{HW} (Theorem 197) or alternatively \cite{khinchin} (Theorem 30). Let $\epsilon \geq 0$, and define $\phi_{\epsilon}(n)=n(\log(n))^{1+\epsilon}.$ Then 
 \beqq
 \sum_{n=1}^{\infty}\frac{1}{n(\log(n))^{1+\epsilon} }
 \eeqq
converges for $\epsilon >0$ and diverges for $\epsilon =0.$ It follows from the above mentioned theorem that 
$$W=\{x\,|\,a_n\geq n(\log(n))\, \text{for infinitely many }\, n\}$$
 is a set of measure one. It also follows that for $\epsilon>0$
 $$ U_{\epsilon}=\{x\,|\,a_n>n(\log(n))^{1+\epsilon}\, \text{for only finitely many $n$}\}$$ has measure one. These together imply the corollary.\qed

\section*{Acknowledgements}
 I would like to thank Kasra Rafi and Doug Hensley for their insightful observations regarding  Corollary \ref{classical}. I am very grateful to the referee for many helpful suggestions and for providing a simple classical proof of Corollary \ref{classical}.

\end{document}